\documentclass[11pt]{article}
\usepackage{amsmath,amssymb,amsfonts,amsthm,mathtools}
\usepackage[colorlinks=true]{hyperref}
\usepackage[nameinlink]{cleveref}
\usepackage[top=25mm,bottom=25mm,left=25mm,right=25mm]{geometry}
\usepackage{microtype}
\usepackage{comment}
\usepackage{xcolor}
\hypersetup{
  colorlinks=true,
  linkcolor=blue,
  citecolor=blue,
  urlcolor=cyan
}
\usepackage{tikz}
\usetikzlibrary{arrows.meta,calc,decorations.pathreplacing}
\theoremstyle{plain}
\newtheorem{theorem}{Theorem}[section]
\newtheorem{lemma}[theorem]{Lemma}
\newtheorem{corollary}[theorem]{Corollary}
\newtheorem{proposition}[theorem]{Proposition}
\theoremstyle{definition}
\newtheorem{definition}[theorem]{Definition}
\newtheorem{remark}[theorem]{Remark}
\newtheorem{problem}[theorem]{Problem}
\newtheorem{example}[theorem]{Example}

\newtheorem{innerthm}{Theorem}
\newenvironment{reptheorem}[1]{%
  \renewcommand\theinnerthm{\ref*{#1}}%
  \innerthm
}{\endinnerthm}

\newcommand{\F}{\mathbb F}
\newcommand{\mA}{\mathcal A}
\newcommand{\mM}{\mathcal M}
\newcommand{\LL}{\mathcal L}
\newcommand{\supp}{\operatorname{supp}}
\newcommand{\wt}{\operatorname{wt}}

\newcommand{\Sim}{\operatorname{Sim}}
\newcommand{\bqbinom}[2]{\genfrac{[}{]}{0pt}{}{#1}{#2}_q}
\DeclareMathOperator{\diam}{diam}
\DeclareMathOperator{\dist}{dist}
\DeclareMathOperator{\rank}{rank}

\title{The Asymptotic Number of Anticodes in Hamming Spaces}
\author{
Anna-Lena Horlemann\thanks{University of St.Gallen, Switzerland. \emph{E-mail}: \texttt{anna-lena.horlemann@unisg.ch}.}
\and
Ananthakrishnan Ravi\thanks{Delft University of Technology, Netherlands. Supported by an NWO open competition grant (OCENW.M.22.090). \emph{E-mail}: \texttt{a.ravi@tudelft.nl}.}
}
\date{}

\begin{document}
\maketitle

\begin{abstract}
In the Hamming space $\mathbb{F}_q^n$, a $d$-anticode is a subset in which every pair of points has Hamming distance at most $d$, or equivalently, a set of diameter at most $d$.
Let $m_q(n,d)$ denote the maximum size of a $d$-anticode and let $\mathcal{A}_q(n,d)$ be the family of all $d$-anticodes.
We show that for fixed $q$, as $n\to\infty$, the total number of $d$-anticodes satisfies
\[
|\mathcal{A}_q(n,d)|=2^{(1+o(1))m_q(n,d)}
\]
for $4\le d\le n-1$, except when $q=2$ and $d=n-1$.

We also count linear $d$-anticodes, that is, subspaces of $\mathbb F_q^n$ with diameter at most $d$.  
For fixed $q$ and $d$, their number is polynomial in $n$. 
For growing $d$, we determine the exponential order as $q^{\Theta_q(d^2+d\log_q(n/d))}$ in the sublinear regime and $q^{\Theta(n^2)}$ when $d/n\to\alpha\in(0,1)$. 
In the near-full regime $n-d=o(n)$, we obtain the sharper count $q^{(1/4+o(1))n^2}$. 
A random-generator construction further improves the binary lower bound for every fixed linear diameter.

\end{abstract}

\section{Introduction}\label{sec:intro}
Anticodes are fundamental objects in coding theory and extremal combinatorics. 
Whereas error-correcting codes are designed to keep codewords sufficiently far apart, anticodes consist of sets of words whose pairwise distances are bounded from above. 
This simple definition gives rise to a rich class of combinatorial structures whose properties depend on the underlying metric and the imposed diameter constraint.

The study of anticodes has a long history, particularly in connection with classical extremal problems. 
A central question has been to determine the maximum cardinality of an anticode of prescribed diameter. 
In Hamming spaces, classical diametric theorems determine the maximum size of such families and describe the geometry of extremal configurations in a wide range of parameter regimes
\cite{Kleitman1966,AhlswedeKatona1977,AhlswedeKhachatrian1998}.

Beyond their intrinsic combinatorial interest, anticodes play an important role in coding theory. In particular, the code--anticode bound relates the minimum distance of a code to the diameter and cardinality of an anticode, generalizing the sphere-packing bound in suitable metric spaces \cite{AhlswedeAydinianKhachatrian2001}.

Alongside these extremal questions, related enumeration questions have been studied for error-correcting codes in various metrics and algebraic settings (see, e.g., \cite{moduleDensity,DongManiZhao2023,densities,tLeeCodes} and references therein).
Such results illustrate the value of enumeration in understanding large families of combinatorial objects. 
However, codes and anticodes satisfy fundamentally different distance constraints, so enumeration results for codes do not directly determine the number of anticodes.

While extremal results determine how large an anticode can be, a different question is how many anticodes satisfy a prescribed diameter constraint. 
In this paper, we study this question asymptotically. 
Asymptotic enumeration can reveal how the number of anticodes grows with the parameters of the ambient space and thus provide a broader perspective on the structure of these families.

\textit{In this work}, we count anticodes in the ambient space $\F_q^n$ (where $q$ is a prime power), equipped with the Hamming distance, which counts the coordinates in which two vectors differ. 
For a set $A\subseteq \F_q^n$, we write $\diam(A)$ for its diameter.
A set of points with diameter at most $d$ is called a \emph{$d$-anticode}.
Let $\mA_q(n,d)$ be the family of all $d$-anticodes. 
We further distinguish between general anticodes and linear ones, i.e., those that form a subspace of $\F_q^n$. 
In particular, linear anticodes form a subclass of general anticodes. 

We first define the maximum possible size of a $d$-anticode in $\F_q^n$ as
\[
  m_q(n,d):=\max\{|A|:A\subseteq \F_q^n,\ \diam(A)\le d\}.
\]
We call a $d$-anticode of size $m_q(n,d)$ a \emph{maximum} $d$-anticode.
The exact value of $m_q(n,d)$ is known from the diametric theorems of Kleitman \cite{Kleitman1966} in the case $q=2$, and Ahlswede--Khachatrian for $q\ge3$ \cite{AhlswedeKhachatrian1998}; see \Cref{thm:kleitman,thm:ahlswede-khachatrian}.
Since every subset of a maximum $d$-anticode is again an anticode, the total number of $d$-anticodes is at least $2^{m_q(n,d)}$.
We show that this lower bound is asymptotically tight for all $4\le d\le n-1$, except when $q=2$ and $d=n-1$.

Let $\mM_q(n,d)$ be the family of \emph{inclusion-maximal} $d$-anticodes.
Since every anticode lies in an inclusion-maximal one, we have
\begin{equation}\label{eq:intro-basic-reduction}
  |\mA_q(n,d)|\le |\mM_q(n,d)|2^{m_q(n,d)}.
\end{equation}
Thus, for general anticodes, we focus on estimating $|\mM_q(n,d)|$.

Our initial approach to count the number of maximal anticodes was to combine the hypergraph container method with stability results for extremal anticodes. 
For $q=2$, Frankl's stability theorem \cite{Frankl2017} makes this approach effective in certain regimes.
For $q\ge3$, stability is known when the codiameter $n-d$ is fixed, but we are not aware of a stability result that applies uniformly over the range of diameters considered here.

Instead, using the Helly-type theorem for Hamming balls of Alon, Jin, and Sudakov \cite{AlonJinSudakov2024}, we show that every inclusion-maximal anticode is determined by at most $2^{d+1}$ of its own points.
This gives our first main result:

\begin{reptheorem}{thm:nonlinear-main}
Fix a prime power $q$, and let $n\to\infty$. Let $d=d(n)$.
For $4\le d\le n-1$, excluding the case $q=2$ and $d=n-1$,
\[
  |\mA_q(n,d)|
  =
  2^{(1+o(1))m_q(n,d)};
\]
for $d\in\{1,2,3\}$,
\[
  |\mA_q(n,d)|=2^{\Theta_q(n)};
\]
and at the endpoint $q=2$ and $d=n-1$,
\[
  |\mA_2(n,n-1)|
  =
  3^{2^{n-1}}
  =
  2^{(\log_2 3)m_2(n,n-1)}.
\]
\end{reptheorem}

We next turn to linear anticodes.
As in the general setting, it is natural to begin with the extremal objects.
In the linear setting, these are the maximum linear $d$-anticodes, namely those attaining the dimension bound $\dim C\le d$.
Their structure follows from the classification of \textit{optimal} linear Hamming anticodes: for $q\ne2$, this was obtained in~\cite{Ravagnani2016}, while for $q=2$ the classification was subsequently obtained in~\cite{DiGiustoGorlaRavagnani2025}.
For counting purposes, the natural parameter turns out to be the support.
We therefore restrict a linear anticode to the coordinates in its support and then count the resulting \textit{full-support} linear codes.
This gives our second result:

\begin{reptheorem}{thm:linear-main}

Fix a prime power $q$, let $n\to\infty$, and let
$0\le d=d(n)\le n$ be an integer. Let
\[
  \LL_q(n,d):=\{C\le\F_q^n:\diam(C)\le d\}.
\]
\begin{enumerate}
  \item If $d$ is fixed, then there is an integer $\sigma_q(d)$ satisfying $d\le \sigma_q(d)\le
    \left\lfloor\frac{qd}{q-1}\right\rfloor$ such that
  \[
    |\LL_q(n,d)|
    =
    \Theta_{q,d}\!\left(n^{\sigma_q(d)}\right).
  \]
  In particular, if $q>d+1$, then $\sigma_q(d)=d$.

  \item If $d\to\infty$ and $d=o(n)$, then
  \[
    |\LL_q(n,d)|
    =
    q^{\Theta_q\!\left(d^2+d\log_q(n/d)\right)}.
  \]
 \item If $d/n\to\alpha$, where $0<\alpha<1$, and
  $\gamma=\min\{1,q\alpha/(q-1)\}$, then
  \[
    q^{(\alpha^2/4+o(1))n^2}
    \le
    |\LL_q(n,d)|
    \le
    q^{(\gamma^2/4+o(1))n^2}.
  \]

  \item If $n-d=o(n)$, then
  \[
    |\LL_q(n,d)|
    =
    q^{(1/4+o(1))n^2}.
  \]
\end{enumerate}

\end{reptheorem}

The lower bound in the linear diameter regime is obtained by taking subspaces supported on exactly $d$ coordinates.
For $q=2$, this choice of support is not optimal.
Random generator matrices allow us to use larger supports while still keeping every codeword's weight at most $d$.
More precisely, we define a function $\phi(\alpha)$ and
prove in \Cref{prop:binary-random-linear-lower} that
\[
  |\LL_2(n,\lfloor\alpha n\rfloor)|
  \ge
  2^{(\phi(\alpha)+o(1))n^2}.
\]
This strictly improves the exponent $\alpha^2/4$ for every
$0<\alpha<1$.

The paper is organized as follows.
In \Cref{sec:prelim}, we collect the extremal, structural, and asymptotic tools used throughout the paper.
In \Cref{sec:nonlinear}, we count general anticodes by passing to inclusion-maximal anticodes and applying the Helly-type theorem for Hamming balls.
In \Cref{sec:linear}, we count linear anticodes by first choosing their support and then counting full-support subspaces on that support.
We treat fixed, sublinear, and linear diameters and the near-full regime separately, and conclude the section with the probabilistic improvement for $q=2$.
Finally, in \Cref{sec:further-directions}, we place our results in broader context, and discuss extensions to growing field size and to codes with restricted distances.

\section{Preliminaries}\label{sec:prelim}

In this section, we collect the building blocks for our counting arguments. 
For general anticodes, we need the classical diametric theorems and the Helly-type theorem for Hamming balls. 
For linear anticodes, we use support bounds together with estimates for Gaussian binomial coefficients. 
We also record the asymptotic estimates that will be used repeatedly.

Let $q$ be a prime power and $\F_q$ the finite field with $q$ elements. 
For $n\ge1$, write $[n]=\{1,\ldots,n\}$, and set $[0]=\varnothing$.
All logarithms without subscripts are natural logarithms. 
We express the counts of general anticodes in base $2$ and those of linear anticodes in base $q$, reflecting the enumeration of subsets and subspaces, respectively.
For fixed $q$, the change of base only rescales logarithms by a constant.
Moreover, we use the convention that $\binom{n}{s}=0$ if $s>n$ or $s<0$.

\begin{definition}
For two vectors $x,y\in \F_q^n$, the \emph{Hamming distance}
$\dist(x,y)$ is the number of coordinates in which they differ.
The \emph{Hamming weight} of $x\in\F_q^n$ is
\[
\wt(x):=\dist(x,0).
\]
The metric space $(\F_q^n,\dist)$ is called a \emph{Hamming space}.
\end{definition}

For $x\in\F_q^n$ and $0\le r\le n$, let
\[
  B(x,r):=\{y\in\F_q^n:\dist(x,y)\le r\}
\]
be the \emph{Hamming ball} of radius $r$ around $x$.
Its size is independent of $x$:
\[
  |B(x,r)|=\sum_{i=0}^{r}\binom ni(q-1)^i.
\]

Next, we state the diametric theorems.
We will use these in the small- and near-full-diameter regimes.
For the rest of our counting argument, we use only the elementary lower bound obtained from a Hamming ball of radius $\lfloor d/2\rfloor$.

\begin{theorem}[Kleitman \cite{Kleitman1966}, Ahlswede--Katona \cite{AhlswedeKatona1977}]
\label{thm:kleitman}
Let $0\le d<n$. In the binary Hamming space $\F_2^n$,
\[
  m_2(n,d)
  =
  \begin{cases}
    \displaystyle \sum_{j=0}^{r}\binom nj,
      & \text{if } d=2r, \\[3mm]
    \displaystyle 2\sum_{j=0}^{r}\binom{n-1}{j},
      & \text{if } d=2r+1.
  \end{cases}
\]
For even $d$, the bound is attained by a Hamming ball of radius $r$.
For odd $d$, it is attained by the union of two Hamming balls of
radius $r$ whose centres are at distance one.
\end{theorem}

Kleitman~\cite{Kleitman1966} proved the even-diameter case. For the
odd-diameter case, see Ahlswede and
Katona~\cite[Theorem~2.5]{AhlswedeKatona1977}; it also follows from
the even-diameter case by Frankl's parity
argument~\cite[Proposition~4.1]{Frankl2017}.

Let $q\ge3$ and, for
$0\le i\le \lfloor d/2\rfloor$, define
\[
  K_i(n,d)
  :=
  \left\{
    x\in\F_q^n:
    \left|\{j\le n-d+2i:x_j\ne0\}\right|\le i
  \right\}.
\]
Thus, the first $n-d+2i$ coordinates lie in a Hamming ball in $\F_q^{n-d+2i}$ of radius $i$ with centre zero,
while the remaining $d-2i$ coordinates are unrestricted.  In particular,
\[
  |K_i(n,d)|
  =
  q^{d-2i}
  \sum_{j=0}^{i}\binom{n-d+2i}{j}(q-1)^j .
\]

\begin{theorem}[Ahlswede--Khachatrian \cite{AhlswedeKhachatrian1998}]\label{thm:ahlswede-khachatrian}
Let $q\ge3$ and $0\le d<n$.  

If $d=n-1$, then
\[
  m_q(n,d)=q^{n-1}=|K_0(n,d)|.
\]

If $0\le d\le n-2$, let
\[
r=
\min\left\{
\left\lfloor\frac d2\right\rfloor,\,
\left\lfloor\frac{n-d-2}{q-2}\right\rfloor
\right\}.
\]
Then
\[
m_q(n,d)
=
|K_r(n,d)|
=
q^{d-2r}
\sum_{j=0}^r
\binom{n-d+2r}{j}(q-1)^j.
\]
\end{theorem}
Ahlswede and Khachatrian state the theorem in terms of the intersection
parameter $t$, where $d=n-t$, and use $q-1$ as the distinguished
alphabet symbol. 
After relabelling this symbol as $0$ and substituting
$t=n-d$, their extremal family $K_i$ becomes exactly the family
$K_i(n,d)$ defined above.

We now distinguish maximum anticodes from anticodes that are maximal only under inclusion.

\begin{definition}
\leavevmode\par
\begin{enumerate}
\item
We call a $d$-anticode $A\subseteq\F_q^n$ \emph{maximum} if
$|A|=m_q(n,d)$.

\item
A $d$-anticode $A\subseteq\F_q^n$ is called
\emph{inclusion-maximal} if no point of $\F_q^n\setminus A$ can be
added to $A$ while keeping diameter at most $d$.
We denote the family of inclusion-maximal $d$-anticodes by
$\mM_q(n,d)$.
\end{enumerate}
\end{definition}
Every maximum anticode is inclusion-maximal, but the converse fails, as the following example shows.
\begin{example}
Consider $\F_2^4$ and the anticode
\[
  A=\{0000,0001,0010,0011\}.
\]
Then $\diam(A)=2$.
Moreover, every vector outside $A$ has a $1$ in one of the first two
coordinates and is therefore at Hamming distance at least $3$ from
some element of $A$.
Consequently, no further vector can be added while preserving
diameter~$2$, so $A$ is inclusion-maximal.

On the other hand, $A$ is not maximum, since
\[
  |A|=4<m_2(4,2)=5.
\]
Indeed, a Hamming ball of radius~$1$ is a $2$-anticode of size~$5$.
\end{example}

For the linear setting, we record a simple upper bound on the
dimension of a linear anticode.

\begin{lemma}\label{lem:linear-dimension-bound}
If $C\le\F_q^n$ is a linear $d$-anticode, then $\dim C\le d$, and equality can be obtained by anticodes of the form $\F_q^d\times\{0\}^{n-d}$.
\end{lemma}

\begin{proof}
Let $k:=\dim C$ and take a generator matrix of $C$ in reduced
row-echelon form. The sum of its rows has weight at least $k$, since
it is nonzero in every pivot coordinate. Hence $k\le d$.
\end{proof}

The structural input for counting inclusion-maximal general anticodes is the following Helly-type theorem.

\begin{theorem}[Alon--Jin--Sudakov, Theorem~1.2 \cite{AlonJinSudakov2024}] \label{thm:AJS}
Let $n>t\ge0$, and let $X$ be a set with $|X|\ge2$.
If $B_1,\ldots,B_m$ are Hamming balls of radius $t$ in $X^n$, then there is a set $I\subseteq[m]$ with $|I|\le2^{t+1}$ such that
\[
  \bigcap_{i=1}^m B_i=\bigcap_{i\in I}B_i.
\]
\end{theorem}

\begin{figure}[htbp]
\centering

\begin{tikzpicture}[
    scale=0.88,
    every node/.style={font=\normalsize},
    pt/.style={circle, inner sep=2pt},
    arrow/.style={-{Latex[length=3.5mm,width=2.7mm]}, line width=1pt}
]

\def\r{2.05}


\begin{scope}[shift={(-4.8,0)}]

    \coordinate (A) at (0,1.35);
    \coordinate (B) at (-1.15,-0.70);
    \coordinate (C) at (1.15,-0.70);

    \coordinate (D) at (-0.55,0.10);
    \coordinate (E) at (0.55,0.10);
    \coordinate (F) at (0,-0.45);

    \draw[draw=purple!70!black, fill=purple!18, fill opacity=0.20, line width=0.9pt]
        (D) circle (\r);
    \draw[draw=teal!70!black, fill=teal!18, fill opacity=0.20, line width=0.9pt]
        (E) circle (\r);
    \draw[draw=red!70!black, fill=red!16, fill opacity=0.20, line width=0.9pt]
        (F) circle (\r);

    \draw[draw=blue!75!black, fill=blue!18, fill opacity=0.28, line width=1pt]
        (A) circle (\r);
    \draw[draw=green!60!black, fill=green!18, fill opacity=0.28, line width=1pt]
        (B) circle (\r);
    \draw[draw=orange!85!black, fill=orange!18, fill opacity=0.28, line width=1pt]
        (C) circle (\r);

    \begin{scope}
        \clip (A) circle (\r);
        \clip (B) circle (\r);
        \clip (C) circle (\r);
        \fill[yellow!70!orange, opacity=0.78] (-3,-3) rectangle (3,3);
    \end{scope}

    \node[pt, fill=blue!75!black]   at (A) {};
    \node[pt, fill=green!60!black]  at (B) {};
    \node[pt, fill=orange!85!black] at (C) {};
    \node[pt, fill=purple!70!black] at (D) {};
    \node[pt, fill=teal!70!black]   at (E) {};
    \node[pt, fill=red!70!black]    at (F) {};

    \node[above=3pt]       at (A) {$x_1$};
    \node[below left=2pt]  at (B) {$x_2$};
    \node[below right=2pt] at (C) {$x_3$};
    \node[left=3pt]        at (D) {$x_4$};
    \node[right=3pt]       at (E) {$x_5$};
    \node[below=3pt]       at (F) {$x_6$};

    \node[font=\large] at (0,0.05) {$A$};

    \node at (0,-4.2)
    {
        $\displaystyle
        A=\bigcap_{i\in[m]} B(x_i,d)
        $
    };

\end{scope}


\draw[arrow] (-0.8,0.12) -- (0.8,0.12);


\begin{scope}[shift={(4.8,0)}]

    \coordinate (A2) at (0,1.35);
    \coordinate (B2) at (-1.15,-0.70);
    \coordinate (C2) at (1.15,-0.70);

    \draw[draw=blue!75!black, fill=blue!18, fill opacity=0.28, line width=1pt]
        (A2) circle (\r);
    \draw[draw=green!60!black, fill=green!18, fill opacity=0.28, line width=1pt]
        (B2) circle (\r);
    \draw[draw=orange!85!black, fill=orange!18, fill opacity=0.28, line width=1pt]
        (C2) circle (\r);

    \begin{scope}
        \clip (A2) circle (\r);
        \clip (B2) circle (\r);
        \clip (C2) circle (\r);
        \fill[yellow!70!orange, opacity=0.78] (-3,-3) rectangle (3,3);
    \end{scope}

    \node[pt, fill=blue!75!black]   at (A2) {};
    \node[pt, fill=green!60!black]  at (B2) {};
    \node[pt, fill=orange!85!black] at (C2) {};

    \node[above=3pt]       at (A2) {$x_1$};
    \node[below left=2pt]  at (B2) {$x_2$};
    \node[below right=2pt] at (C2) {$x_3$};

    \node[font=\large] at (0,0.05) {$A$};

    \node at (0,-4.2)
    {
        $\displaystyle
        A=\bigcap_{i\in I} B(x_i,d)
        $
    };

\end{scope}

\end{tikzpicture}

\caption{
A schematic illustration of the Helly-type theorem for Hamming balls. }
\label{fig:helly-reduction-colored}

\end{figure}
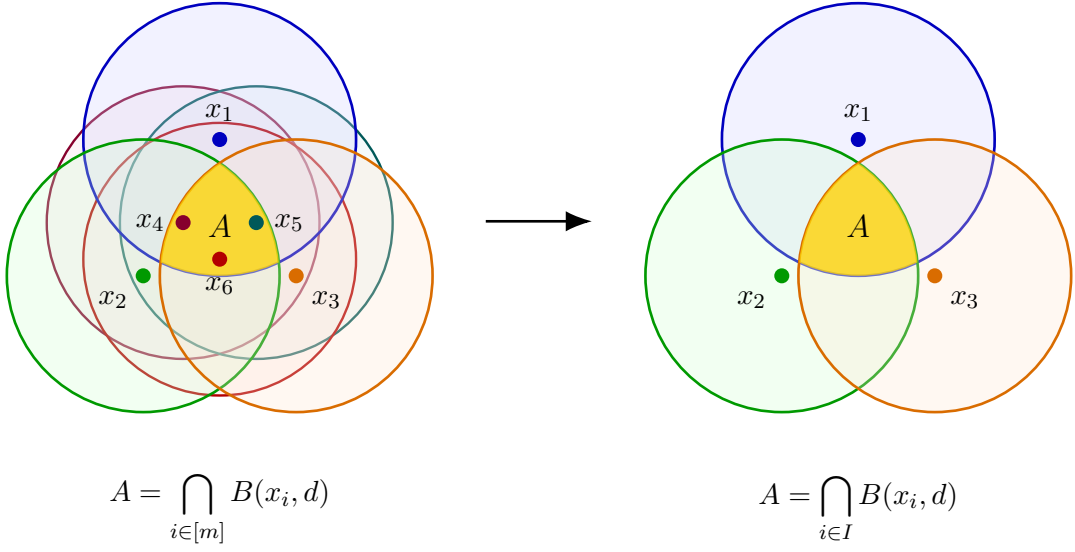
We will also use the following bound for independent sets in regular graphs.

\begin{theorem}[Zhao, Corollary~1.3 \cite{Zhao2010}]\label{thm:zhao}
Let $G$ be an $N$-vertex, $D$-regular graph with $D\ge1$.
If $i(G)$ denotes the number of independent sets of $G$, then
\[
  i(G)
  \le
  \left(2^{D+1}-1\right)^{N/(2D)}.
\]
\end{theorem}
Unless stated otherwise, $q$ is fixed and asymptotics are taken as $n\to\infty$.
We use subscripts on Landau notation whenever more than one parameter may vary.
We use $O_{q,d}(\cdot)$, $\Omega_{q,d}(\cdot)$, and
$\Theta_{q,d}(\cdot)$ when the implicit constants may depend on
$q$ and $d$, but not on $n$.
Similarly, $O_q(\cdot)$, $\Omega_q(\cdot)$, and
$\Theta_q(\cdot)$ allow dependence on $q$ only.

We will use the following binomial estimates in our calculations.
\begin{lemma}\label{lem:binomial-estimates}
For $1\le k\le n$,
\[
  \left(\frac nk\right)^k
  \le \binom nk
  \le \left(\frac{en}{k}\right)^k.
\]
If $k$ is fixed and $n\to\infty$, then
\[
  \binom nk
  =
  \frac{n^k}{k!}+O_k(n^{k-1}).
\]
If $k=o(n)$, then
\[
  \log_2\binom nk
  =
  (1+o(1))k\log_2\frac nk.
\]
If $k/n\to\rho\in(0,1)$, then
\[
  \log_2\binom nk
  =
  (h(\rho)+o(1))n,
\]
where
$
  h(\rho)
  :=
  -\rho\log_2\rho-(1-\rho)\log_2(1-\rho)
$ denotes the binary entropy function evaluated at $\rho$.
\end{lemma}
These standard estimates are collected, for example, in
\cite{Das2016Binomial}.
Since $q$ is fixed, changing the base of the logarithms gives the corresponding estimates with $\log_q$.
We use the convention $h(0)=h(1)=0$.

The preceding estimates give the following bound on a Hamming ball. 

\begin{lemma}\label{lem:hamming-ball-estimate}
For $x\in\F_q^n$ and $1\le r\le n$,
\[
  |B(x,r)|
  \le
  (r+1)\left(\frac{e(q-1)n}{r}\right)^r.
\]
\end{lemma}

\begin{proof}
Choose $0\le j\le r$ for which
$\binom nj(q-1)^j$ is largest. Then
\[
  |B(x,r)|
  \le
  (r+1)\binom nj(q-1)^j.
\]
If $j\ge1$, then \Cref{lem:binomial-estimates} gives
\[
  \binom nj(q-1)^j
  \le
  \left(\frac{e(q-1)n}{j}\right)^j
  \le
  \left(\frac{e(q-1)n}{r}\right)^r,
\]
where the last inequality follows because
$t\mapsto (e(q-1)n/t)^t$ is non-decreasing on $1\le t\le n$.
The case $j=0$ is immediate. Hence
\[
  |B(x,r)|
  \le
  (r+1)\left(\frac{e(q-1)n}{r}\right)^r.
\]
\end{proof}

For integers $0\le r\le n$, the Gaussian binomial coefficient $\bqbinom{n}{r}$ counts the $r$-dimensional subspaces of $\F_q^n$
and is given by 
\begin{equation}\label{eq:gaussian-binomial}
  \bqbinom{n}{r}
  =
  \prod_{i=0}^{r-1}
  \frac{q^{n-i}-1}{q^{r-i}-1}.
\end{equation}
We will use the following estimates (also see \cite{Ihringer2015}) in our calculations.

\begin{lemma}\label{lem:gaussian-uniform}
For every prime power $q$, let $C_q:=\prod_{j=1}^{\infty}(1-q^{-j})^{-1}$.
Then $C_q<\infty$, and for all integers $0\le r\le n$,
\[
    q^{r(n-r)}
    \le
    \bqbinom nr
    \le
    C_q\,q^{r(n-r)}.
\]
Moreover, if $q=q(n)\to\infty$, then uniformly for $0\le r\le n$,
\[
    \bqbinom nr=(1+o(1))q^{r(n-r)}.
\]
\end{lemma}

\begin{proof}
From \Cref{eq:gaussian-binomial} we obtain.
\begin{equation}\label{eq:gaussian-binomial-2}
       \bqbinom nr
    =
    q^{r(n-r)}
    \prod_{j=1}^{r}
    \frac{1-q^{-(n-r+j)}}{1-q^{-j}}. 
\end{equation}
Since $n-r+j\ge j$, every factor in the product is at least $1$, while $\frac{1-q^{-(n-r+j)}}{1-q^{-j}}
    \le (1-q^{-j})^{-1}$.
Hence
\[
    1
    \le
    \frac{\bqbinom nr}{q^{r(n-r)}}
    \le
    \prod_{j=1}^{r}(1-q^{-j})^{-1}
    \le C_q.
\]
The product defining $C_q$ converges because, for $0\le x\le1/2$, $-\log(1-x)\le 2x$.
Hence
\[
    0\le \log C_q
    =
    \sum_{j=1}^{\infty}-\log(1-q^{-j})
    \le
    2\sum_{j=1}^{\infty}q^{-j}
    =
    \frac{2}{q-1}.
\]
In particular, $C_q=1+o(1)$ as $q\to\infty$. 
\end{proof}

\begin{lemma}\label{lem:gaussian-estimates}
Fix a prime power $q$.

\begin{enumerate}

  \item[(i)] As $t\to\infty$,
  \[
    \sum_{r=0}^{t}\bqbinom tr
    =
    q^{t^2/4+O_q(\log t)}.
  \]

  \item[(ii)] If $r_t=\lfloor t/2\rfloor$, then, as $t\to\infty$,
  \[
    \bqbinom t{r_t}
    -
    t\bqbinom {t-1}{r_t}
    =
    (1-o(1))\bqbinom t{r_t}
    =
    \Theta_q\!\left(q^{t^2/4}\right).
  \]
\end{enumerate}
\end{lemma}

\begin{proof}
For part~(i), \Cref{lem:gaussian-uniform} gives
\[
  q^{\lfloor t^2/4\rfloor}
  \le
  \sum_{r=0}^{t}\bqbinom tr
  \le
  (t+1)C_q q^{t^2/4}.
\]
For part~(ii), \Cref{eq:gaussian-binomial-2} gives
\[
  \frac{\bqbinom {t-1}{r_t}}{\bqbinom t{r_t}}
  =
  \frac{q^{t-r_t}-1}{q^t-1}.
\]
Hence
\[
  \frac{t\bqbinom {t-1}{r_t}}{\bqbinom t{r_t}}
  \le
  \frac{tq^{-r_t}}{1-q^{-t}}
  =
  o(1).
\]
Therefore
\[
  \bqbinom t{r_t}
  -
  t\bqbinom {t-1}{r_t}
  =
  (1-o(1))\bqbinom t{r_t}.
\]
Finally, \Cref{lem:gaussian-uniform} and
$r_t(t-r_t)=\lfloor t^2/4\rfloor$ imply $\bqbinom t{r_t}   =   \Theta_q\!\left(q^{t^2/4}\right)$.

\end{proof}

The random-generator construction in the final part of the paper will also require a binomial upper-tail estimate.
For integers $s\ge1$ and $0\le p\le1$, we write
\[
  X\sim\operatorname{Bin}(s,p)
\]
when $X$ is the sum of $s$ independent Bernoulli random variables
with success probability $p$.
For $a,p\in(0,1)$, define the binary relative entropy
\[
  D_2(a\|p)
  :=
  a\log_2\frac ap
  +(1-a)\log_2\frac{1-a}{1-p}.
\]
\begin{lemma}\label{lem:binomial-tail}
Let $X\sim\operatorname{Bin}(s,p)$ and let $0<p<a<1$.
Then
\[
  \mathbb P(X\ge as)
  \le
  2^{-sD_2(a\|p)}.
\]
\end{lemma}

\begin{proof}
Write $X=X_1+\cdots+X_s$, where $X_1,\ldots,X_s$ are independent
$\operatorname{Bernoulli}(p)$ random variables and apply the Chernoff
bound. 
\end{proof}

\section{The general count}\label{sec:nonlinear}

We will now derive an asymptotic count for the number of $d$-anticodes in $\F_q^n$. Recall that $\mA_q(n,d)$ denotes the family of all $d$-anticodes in $\F_q^n$. 
Since $\F_q^n$ is finite for fixed $n$, every $d$-anticode is contained in an inclusion-maximal one.
Hence
\begin{equation}\label{eq:basic-reduction}
  2^{m_q(n,d)} \le |\mA_q(n,d)|
  \le \sum_{S\in\mM_q(n,d)}2^{|S|}
  \le |\mM_q(n,d)|2^{m_q(n,d)}.
\end{equation}
The lower bound $|\mA_q(n,d)|\ge 2^{m_q(n,d)}$ follows from counting all subsets of one maximum $d$-anticode (which is always known to exist). 
The upper bound leaves us with the problem of counting inclusion-maximal anticodes, i.e., determining $|\mM_q(n,d)|$. 
The next lemma gives a description of each such anticode.

\begin{lemma}\label{lem:helly-basis}
Let $0\le d<n$.
If $S\in\mM_q(n,d)$, then there is a set $T\subseteq S$ with $|T|\le2^{d+1}$ such that
\[
  S=\bigcap_{x\in T}B(x,d).
\]
\end{lemma}

\begin{proof}
We first prove that
\begin{equation}\label{eq:maximal-as-intersection}
  S=\bigcap_{x\in S}B(x,d).
\end{equation}
The inclusion $S\subseteq\bigcap_{x\in S}B(x,d)$ follows from $\diam(S)\le d$.
Conversely, let $y\in\bigcap_{x\in S}B(x,d)$.
Then $\dist(y,x)\le d$ for every $x\in S$, so $S\cup\{y\}$ still has diameter at most $d$.
By inclusion-maximality, $y\in S$.
This proves \eqref{eq:maximal-as-intersection}.
Now apply \Cref{thm:AJS} to the family $\{B(x,d):x\in S\}$.
We get $T\subseteq S$ with $|T|\le2^{d+1}$, and \eqref{eq:maximal-as-intersection} gives the claim.
\end{proof}

\begin{lemma}\label{lem:maximal-count}
For every $0\le d<n$ and every $x\in\F_q^n$,
\[
  |\mM_q(n,d)|
  \le
  q^n|B(x,d)|^{2^{d+1}-1}.
\]
Consequently,
\[
  \log_2|\mM_q(n,d)|
  \le
  n\log_2q
  +(2^{d+1}-1)\log_2|B(x,d)|.
\]
\end{lemma}

\begin{proof}
By \Cref{lem:helly-basis}, every $S\in\mM_q(n,d)$ is determined by an ordered list of at most $2^{d+1}$ points.
By repeating elements if necessary, we may regard it as an ordered list of exactly $2^{d+1}$ points.
The first point has $q^n$ choices.
Every subsequent point lies in a radius-$d$ ball around the first, and hence has at most $|B(x,d)|$ choices.
\end{proof}

We shall use two elementary lower bounds for $m_q(n,d)$. 
A Hamming ball of radius $\lfloor d/2\rfloor$ has diameter at most $d$.
This gives
\begin{equation}\label{eq:ball-lower}
  m_q(n,d)\ge |B(x,\lfloor d/2\rfloor)|.
\end{equation}

We can now compare the maximal-anticode count with $m_q(n,d)$ in the sublinear and linear-diameter regimes.

\begin{proposition}\label{lem:nonlinear-bulk}
Fix a prime power $q$, and let $n\to\infty$. Then
\[
  |\mA_q(n,d)|
  =
  2^{(1+o(1))m_q(n,d)}
\]
in each of the following regimes:
\begin{enumerate}
  \item $4\le d=o(n)$;
  \item $d/n\to\delta$ for some $0<\delta<1$.
\end{enumerate}
\end{proposition}

\begin{proof}
By \eqref{eq:basic-reduction}, it suffices to prove that
$\log_2|\mM_q(n,d)|=o\!\left(m_q(n,d)\right).$ 
Let $r:=\left\lfloor\frac d2\right\rfloor$.
By \eqref{eq:ball-lower}, we have
\begin{equation}\label{eq:mq-ball-term-lower}
m_q(n,d)
\ge |B(x,r)|
\ge \binom nr(q-1)^r
\ge \binom nr.
\end{equation}
We prove the two diameter regimes separately:

\paragraph{1. Sublinear diameter.}
Suppose that $4\le d=o(n)$. 
By \Cref{lem:maximal-count,lem:hamming-ball-estimate},
\[
\begin{aligned}
  \log_2|\mM_q(n,d)|
  &\le n\log_2 q
  +(2^{d+1}-1)\log_2|B(x,d)|\\
  &\le n\log_2 q
  +(2^{d+1}-1)
  \left(
    \log_2(d+1)
    +d\log_2\frac{e(q-1)n}{d}
  \right),
\end{aligned}
\]
and hence
\[
  \log_2|\mM_q(n,d)|
  =
  O_q\!\left(n+2^d d\log\frac nd\right).
\]
By \eqref{eq:mq-ball-term-lower} and
\Cref{lem:binomial-estimates},
\[
m_q(n,d)
\ge \binom nr
\ge \left(\frac nr\right)^r
\ge \left(\frac{2n}{d}\right)^r.
\]

We compare each term with this lower bound.
Since $d\in\{2r,2r+1\}$ and $r\ge2$, we have
$d\le2^{r+1}$. 
Therefore
\[
  \frac{n}{(2n/d)^r}
  =
  \frac{d}{2^r(n/d)^{r-1}}
  \le
  \frac{2}{(n/d)^{r-1}}
  \le
  \frac{2d}{n}
  \longrightarrow0.
\]
Also, since $d-r\le r+1$,
\[
\frac{2^d d\log(n/d)}{(2n/d)^r} = 2^{d-r}d\left(\log\frac nd\right)
  \left(\frac dn\right)^r \le 2^{r+1}d\left(\log\frac nd\right)
  \left(\frac dn\right)^r \le 4\left(\log\frac nd\right)
  \left(\frac{4d}{n}\right)^r
\]
For all sufficiently large $n$, we have $4d/n<1$, and since $r\ge2$,
\[
  4\log\frac nd
  \left(\frac{4d}{n}\right)^r
  \le
  64\left(\frac dn\right)^2\log\frac nd
  \longrightarrow0.
\]
Hence
\[
  n+2^d d\log\frac nd
  =
  o\!\left(\left(\frac{2n}{d}\right)^r\right)
  =
  o\!\left(m_q(n,d)\right).
\]

\paragraph{2. Linear diameter.}
Suppose that $\frac{d}{n}\longrightarrow\delta$ for some $0<\delta<1$, and let $r:=\lfloor d/2\rfloor$.
Since $r/n\to\delta/2$, \eqref{eq:mq-ball-term-lower} and
\Cref{lem:binomial-estimates} give
\[
m_q(n,d)
\ge
\binom nr
\ge
2^{(h(\delta/2)+o(1))n}.
\]
Strict concavity of $h$, together with
$h(0)=0$ and $h(1/2)=1$, gives $h(\delta/2)>\delta$.

On the other hand, $|B(x,d)|\le q^n$, so
\[
  \log_2|\mM_q(n,d)|
  \le
  2^{d+1}n\log_2q
  =
  2^{(\delta+o(1))n}.
\]
Consequently,
\[
  \frac{\log_2|\mM_q(n,d)|}{m_q(n,d)}
  \le
  2^{-(h(\delta/2)-\delta+o(1))n}
  \longrightarrow0.
\]

\end{proof}

The preceding proposition requires $d\geq 4$. 
We now give an argument for the diameters $d=1,2,3$.

\begin{proposition}\label{lem:nonlinear-small-diameter}
Fix a prime power $q$. For every $d\in\{1,2,3\}$,
\[
    2^{m_q(n,d)}
    \le |\mA_q(n,d)|
    \le
    2^{m_q(n,d)+n\log_2q+O_q(\log n)}.
\]
Moreover,
\[
    |\mA_q(n,1)|
    =
    1+q^n+nq^{n-1}(2^q-q-1),
\]
and in particular,
\[
    |\mA_q(n,d)|=2^{\Theta_q(n)}
    \qquad
    \text{for }d\in\{1,2,3\}.
\]
\end{proposition}

\begin{proof}
Let $d\in\{1,2,3\}$. Since $d$ and $q$ are fixed,
$
    |B(x,d)|=O_q(n^d).
$
Hence \Cref{lem:maximal-count} gives
\[
    \log_2|\mM_q(n,d)|
    \le n\log_2q+O_q(\log n).
\]
Together with \eqref{eq:basic-reduction}, this gives
\[
    2^{m_q(n,d)}
    \le |\mA_q(n,d)|
    \le
    2^{m_q(n,d)+n\log_2q+O_q(\log n)}.
\]

For $d=1$, every anticode containing at least two points is contained
in a unique coordinate line: if $x\ne y$ have distance at most $1$,
then they differ in exactly one coordinate, and every point at distance
at most $1$ from both must vary only in that coordinate.
There are $nq^{n-1}$ coordinate lines, each containing $q$ points.
Thus, counting the empty set, the $q^n$ singletons, and the subsets of
size at least two in each coordinate line gives
\[
    |\mA_q(n,1)|
    =
    1+q^n+nq^{n-1}(2^q-q-1)
    =
    2^{\Theta_q(n)}.
\]

Finally, \Cref{thm:kleitman,thm:ahlswede-khachatrian} give, for all
sufficiently large $n$,
\[
    m_q(n,2)=1+(q-1)n,
    \qquad
    m_q(n,3)=q\bigl(1+(q-1)(n-1)\bigr).
\]
Hence $m_q(n,d)=\Theta_q(n)$ for $d\in\{2,3\}$, and the preceding
bounds give
$
    |\mA_q(n,d)|=2^{\Theta_q(n)}
$
also for $d\in\{2,3\}$.
\end{proof}

We next consider the diameter regime $d=n-o(n)$, treating $q\ge3$ and $q=2$ separately. 
We call this the \textit{near-full} regime.
We first record a lower bound on $m_q(n,d)$.

\begin{lemma}\label{lem:coordinate-subcube-lower}
For every $0\le d\le n$,
\[
  m_q(n,d)\ge q^d.
\]
\end{lemma}

\begin{proof}
Fix $n-d$ coordinates and let the remaining $d$ coordinates vary freely.
The resulting set has size $q^d$, and any two of its points differ in at
most those $d$ coordinates.
\end{proof}

\begin{proposition}\label{lem:nonlinear-near-full-qary}
Fix a prime power $q\ge3$. 
Let $d=n-t$, where $1\le t=o(n)$. 
Then
\[
  |\mA_q(n,d)|
  =
  2^{(1+o(1))m_q(n,d)}.
\]
\end{proposition}

\begin{proof}
By \Cref{lem:maximal-count} and the trivial estimate $|B(x,d)|\le q^n$,
\[
  \log_2|\mM_q(n,d)|
  =
  O_q\!\left(n2^d\right).
\]
On the other hand, \Cref{lem:coordinate-subcube-lower} gives
$
  m_q(n,d)\ge q^d.
$
Since $d=n-o(n)$ and $q\ge3$,
\[
  \frac{\log_2|\mM_q(n,d)|}{m_q(n,d)}
  \le
  O_q\!\left(n\left(\frac2q\right)^d\right)
  \longrightarrow0.
\]
The result now follows from \eqref{eq:basic-reduction}.
\end{proof}

For $q=2$ in the near-full regime, we will view the diameter condition as an independent-set condition.
For $1\le t\le n$, let $G_{n,t}$ be the graph on $\F_2^n$ in which
$x$ and $y$ are adjacent when $\dist(x,y)>n-t$.  
Then the $(n-t)$-anticodes are exactly the independent sets of $G_{n,t}$.
\begin{lemma}\label{lem:binary-forbidden-graph}
$G_{n,t}$ is a regular graph of degree
$
  D_{n,t}
  :=
  \sum_{j=0}^{t-1}\binom nj.
$
\end{lemma}

\begin{proof}
For fixed $x$, its neighbors are exactly the vectors at distances $n-t+1,\ldots,n$ from $x$.

\[
  \deg(x)
  =
  \sum_{j=n-t+1}^{n}\binom nj
  =
  \sum_{j=0}^{t-1}\binom nj.
\]
This is independent of $x$.
\end{proof}

We will also need the following estimate for the size of a maximum binary anticode in the near-full regime. 

\begin{lemma}\label{lem:binary-near-full-extremal}
Let $t=n-d$ satisfy $1\le t=o(\sqrt n)$. Then
\[
    m_2(n,d)
    =
    \left(\frac12+o(1)\right)2^n.
\]
\end{lemma}

\begin{proof}
Suppose first that $d$ is even, say $d=2r$. By
\Cref{thm:kleitman},
\[
    m_2(n,d)=\sum_{i=0}^{r}\binom ni.
\]
Let $M=\lfloor n/2\rfloor$. By symmetry,
\[
    \sum_{i=0}^{M}\binom ni
    =
    2^{n-1}+O\left(\binom nM\right).
\]
Since $d=n-t$, we have $M-r=O(t)$. Hence
\[
    m_2(n,d)
    =
    2^{n-1}
    +
    O\left(t\binom nM\right).
\]
Using the standard estimate (see \cite{Das2016Binomial}) $\binom nM=O\left(\frac{2^n}{\sqrt n}\right)$, 
we obtain
\[
    m_2(n,d)
    =
    2^{n-1}
    +
    O\left(\frac{t}{\sqrt n}2^n\right)
    =
    \left(\frac12+o(1)\right)2^n.
\]
The odd-diameter case follows analogously from the formula in \Cref{thm:kleitman}.
\end{proof}

\begin{proposition}\label{lem:nonlinear-near-full-binary}
Let $d=n-t$, where $2\le t=o(n)$.
Then
\[
  |\mA_2(n,d)|
  =
  2^{(1+o(1))m_2(n,d)}.
\]
\end{proposition}

\begin{proof}
Fix $0<\varepsilon<1/2$ and split according to whether $t\le n^\varepsilon$ or $t>n^\varepsilon$. 
If $2\le t\le n^\varepsilon$, then $t=o(\sqrt n)$, so by \Cref{lem:binary-near-full-extremal}, $m_2(n,d)=\left(\frac12+o(1)\right)2^n$.
Since $\mA_2(n,d)$ is the family of independent sets of the $D_{n,t}$-regular graph $G_{n,t}$ on $2^n$ vertices, and $D_{n,t}\ge n+1$ for $t\ge2$, \Cref{thm:zhao} gives
\[
\log_2|\mA_2(n,n-t)|
\le
\frac{2^n}{2D_{n,t}}\log_2(2^{D_{n,t}+1}-1)
\le
2^{n-1}+\frac{2^n}{2D_{n,t}}
=
\left(\frac12+o(1)\right)2^n.
\]
Together with $|\mA_2(n,d)|\ge 2^{m_2(n,d)}$, this gives
$|\mA_2(n,d)|=2^{(1+o(1))m_2(n,d)}$.

Suppose now that $t>n^\varepsilon$. By \Cref{lem:maximal-count} and $|B(x,d)|\le2^n$,
\[
\log_2|\mM_2(n,n-t)|=O(n2^{n-t}).
\]
Let $M=\lfloor n/2\rfloor$, $r=\lfloor(n-t)/2\rfloor$, and $s=M-r=O(t)$. Since $\binom nM$ is the largest binomial coefficient, $\binom nM\ge \frac{2^n}{n+1}$, while
\[
\frac{\binom nr}{\binom nM}
=
\prod_{k=r+1}^{M}\frac{k}{n-k+1}
\ge
\left(\frac{r+1}{n-r}\right)^s
=
\left(1-O\left(\frac tn\right)\right)^s
=
2^{-O(t^2/n)}.
\]
Hence, by \eqref{eq:ball-lower},
\[
m_2(n,d)\ge \binom nr\ge 2^{\,n-O(t^2/n)-O(\log n)}.
\]
Therefore
\[
\frac{\log_2|\mM_2(n,d)|}{m_2(n,d)}
\le
2^{-t+O(t^2/n)+O(\log n)}.
\]
Since $t=o(n)$ we have $t^2/n=o(t)$. Since $t>n^\varepsilon$, we also have $\log n=o(t)$. 
Thus,
\[
\frac{\log_2|\mM_2(n,d)|}{m_2(n,d)}\to0,
\]
and the result follows from \eqref{eq:basic-reduction}.
\end{proof}

\begin{proposition}\label{lem:binary-endpoint}
At the binary endpoint $d=n-1$,
$
  |\mA_2(n,n-1)|
  =
  3^{2^{n-1}},
$
and
$
  m_2(n,n-1)=2^{n-1}.
$
Consequently,
\[
  |\mA_2(n,n-1)|
  =
  2^{(\log_2 3)\,m_2(n,n-1)}.
\]
\end{proposition}

\begin{proof}
For every $x\in\F_2^n$, there is a unique vector at distance $n$ from $x$, namely its complement $\bar x$. Thus an $(n-1)$-anticode is exactly a set containing no complementary pair.

The space $\F_2^n$ is partitioned into $2^{n-1}$ complementary pairs
$\{x,\bar x\}$. From each pair, an $(n-1)$-anticode may contain
neither element, or exactly one of the two. Hence there are three
independent choices for each pair, giving
\[
  |\mA_2(n,n-1)|
  =
  3^{2^{n-1}}.
\]

A maximum $(n-1)$-anticode contains exactly one element from each
complementary pair. Therefore
$
  m_2(n,n-1)=2^{n-1}.
$
\end{proof}

We now collect the preceding estimates into our main theorem on the number of general anticodes.

\begin{theorem}\label{thm:nonlinear-main}
Fix a prime power $q$, let $d=d(n)$, and let $n\to\infty$.

\begin{itemize}
    \item 
If $4\le d\le n-1$, excluding the case $q=2$ and $d=n-1$, then
\[
  |\mA_q(n,d)|
  =
  2^{(1+o(1))m_q(n,d)}.
\]
 \item
For $d\in\{1,2,3\}$, we have
\[
  |\mA_q(n,d)|=2^{\Theta_q(n)}.
\]
\item
At the exceptional binary endpoint $d=n-1$, we have
\[
  |\mA_2(n,n-1)|
  =
  3^{2^{n-1}}
  =
  2^{(\log_2 3)m_2(n,n-1)}.
\]
\end{itemize}
\end{theorem}

\begin{proof}
The assertions for $d\in\{1,2,3\}$ and for the binary endpoint
$d=n-1$ follow from
\Cref{lem:nonlinear-small-diameter,lem:binary-endpoint}.

It remains to prove the first assertion. Consider any subsequence of
$n$. Since $d/n\in[0,1]$, there is a further subsequence along which $\frac dn\longrightarrow\delta$ for some $\delta\in[0,1]$.

If $\delta=0$, then $d=o(n)$, so
\Cref{lem:nonlinear-bulk} applies. If $0<\delta<1$, the second part of
\Cref{lem:nonlinear-bulk} applies. If $\delta=1$, then
$n-d=o(n)$, and the result follows from
\Cref{lem:nonlinear-near-full-qary,lem:nonlinear-near-full-binary}. 
Thus every subsequence has a further subsequence along which
\[
\frac{\log_2|\mA_q(n,d)|}{m_q(n,d)}\longrightarrow1.
\]
Hence the full sequence converges to $1$, proving the claim.
\end{proof}

\section{The linear count}\label{sec:linear}

For a linear subspace $C\le\F_q^n$, the diameter condition has a simple form.
Since $C-C=C$, we have
\[
  \diam(C)=\max_{x,y\in C}\dist(x,y)=\max_{c\in C}\wt(c).
\]
Thus $\diam(C)\le d$ means exactly that every word of $C$ has Hamming weight at most $d$.
Analogous to the general setting, we denote by 
\[
  \LL_q(n,d):=\{C\le\F_q^n:\diam(C)\le d\}
\]
the set of all linear $d$-anticodes in $\F_q^n$. 
We will extensively use the notion of \emph{support} for counting linear anticodes, which is defined as 
\[
  \supp(C):=\{i\in[n]:\exists c\in C\text{ with }c_i\ne0\}
\]
for some $C\le\F_q^n$. 
We say that $C\le\F_q^n$ has \emph{full support} if $\supp(C)=[n]$.
Equivalently, no coordinate is identically zero on $C$.
Every linear code has full support after restriction to its support. 
Therefore, our strategy is to first choose the support $S\subseteq \{1,\dots,n\}$ of size $s$, and then count full-support codes on that support.

\begin{lemma}\label{lem:full-support-upper}
If $C\le\F_q^s$ has full support and $\diam(C)\le d$, then
\[
  s\le \frac{qd}{q-1}.
\]
\end{lemma}

\begin{proof}
Suppose $C\le\F_q^s$ has full support and dimension $k$.
For each coordinate $i\in[s]$, the map
\[
  C\to\F_q,\qquad c\mapsto c_i
\]
is a non-zero linear functional.
Hence it is non-zero on exactly $(q-1)q^{k-1}$ words of $C$.
Summing this over all coordinates gives
\[
  \sum_{c\in C}\wt(c)=s(q-1)q^{k-1}.
\]
Since $|C|=q^k$, the average weight of a word of $C$ is
\[
  \frac1{|C|}\sum_{c\in C}\wt(c)=\frac{s(q-1)}q.
\]
If $\diam(C)\le d$, then every word has weight at most $d$, and therefore $\frac{s(q-1)}{q} \le d$. 
\end{proof}

\begin{remark}
The support bound is a Plotkin-type averaging argument, analogous to the classical Plotkin bound; see \cite{HallNotes}.
\end{remark}

%
Lemma \ref{lem:full-support-upper} shows that every linear $d$-anticode is supported on at most $\lfloor  dq/(q-1)\rfloor$ coordinates, which motivates counting anticodes first by their support size. 
Let \(a_q(s,d)\) be the number of full-support linear \(d\)-anticodes in \(\mathbb F_q^s\).
\[
  a_q(s,d)
  :=
  \left|
  \left\{
  C\le\F_q^s:
  \supp(C)=[s],\ \diam(C)\le d
  \right\}
  \right|
\]
for integers $s,d\ge0$. 

\begin{proposition}\label{thm:linear-exact}
For every prime power $q$ and every $0\le d<n$,
\[
  |\LL_q(n,d)|
  =
  \sum_{s=0}^{\lfloor dq/(q-1)\rfloor}
  \binom ns a_q(s,d)  \le
  \sum_{s=0}^{\lfloor dq/(q-1)\rfloor}
  \binom ns
  \sum_{r=0}^{s}\bqbinom sr.
\]
\end{proposition}

\begin{proof}

Puncturing a code outside its support \(S\) gives a full-support linear $d$-anticode in $\F_q^S$, and conversely every such code extends uniquely by zeros outside $S$. 
Hence, the codes with support size $s$ are in bijection with pairs $(S,C')$ with $|S|=s$ and $\supp(C')=S$. As there are $\binom{n}{s}$ choices for the support, we get
\[
  |\LL_q(n,d)|
  =
  \sum_{s=0}^{\lfloor dq/(q-1)\rfloor}
  \binom ns a_q(s,d)
\]
The inequality follows bounding above by $a_q(s,d)$ with the total number of subspaces of $\F_q^S$.
\end{proof}

The exact formula gives the required upper bound once all subspaces on a fixed support are counted. 
For the lower bound, we use subspaces supported on exactly $d$ coordinates.

\begin{lemma}\label{lem:coordinate-lower}
For every $1\le d<n$,
\[
  |\LL_q(n,d)|
  \ge
  \binom nd
  \left(
    \bqbinom d{\lfloor d/2\rfloor}
    -d\bqbinom {d-1}{\lfloor d/2\rfloor}
  \right).
\]
\end{lemma}

\begin{proof}
Choose a set $S\subseteq[n]$ with $|S|=d$.
Every subspace of $\F_q^S$ has all its words supported inside $S$.
Hence, every word has weight at most $d$, so every subspace of $\F_q^S$, extended by zero outside $S$, is a linear $d$-anticode in $\F_q^n$.
For a fixed support $S$ of size $d$, the number of $r$-dimensional subspaces of
$\F_q^S$ is $\bqbinom dr$, and this is largest when $r$ is closest to $d/2$.
So we choose $r=\lfloor d/2\rfloor$ to maximize this number.

A subspace does not have support exactly $S$ precisely when one of the coordinates in $S$ is identically zero on it.
Equivalently, it is contained in one of the $d$ coordinate hyperplanes
\[
  H_i=\{x\in\F_q^S:x_i=0\}.
\]
Each $H_i$ is isomorphic to $\F_q^{d-1}$, and therefore contains $\bqbinom {d-1}r$ $r$-dimensional subspaces.
By the union bound, the number of $r$-dimensional subspaces of $\F_q^S$ which do not have full support is at most $d\bqbinom {d-1}r$.
Therefore at least
\[
  \bqbinom dr-d\bqbinom {d-1}r
\]
$r$-dimensional subspaces have support exactly $S$.
Thus different choices of $S$ give different codes.
Since there are $\binom nd$ choices for $S$, we get
\[
  |\LL_q(n,d)|
  \ge
  \binom nd
  \left(
    \bqbinom dr-d\bqbinom {d-1}r
  \right).
\]
Substituting $r=\lfloor d/2\rfloor$ gives the claim.
\end{proof}

We will use \Cref{thm:linear-exact} and \Cref{lem:coordinate-lower} in \Cref{thm:linear-main} to derive asymptotic estimates and bounds on the number of linear $d$-anticodes. We first treat the near-full regime separately.

\begin{proposition}\label{prop:linear-near-full}
Fix a prime power $q$, and let $d=n-t$, where $0 \le t=o(n)$.
Then
\[
  |\LL_q(n,d)|
  =
  q^{(1/4+o(1))n^2}.
\]
\end{proposition}

\begin{proof}
As an upper bound, we simply use all subspaces of $\F_q^n$, which gives us (together with \Cref{lem:gaussian-estimates})
\[
  |\LL_q(n,d)|
 \le
  \sum_{r=0}^{n}\bqbinom nr
  =
  q^{n^2/4+O_q(\log n)}  \le
  q^{(1/4+o(1))n^2}.
\]

For a lower bound, fix one set $S\subseteq[n]$ with $|S|=d$.
Every subspace of $\F_q^S$, extended by zero outside $S$, is a linear
$d$-anticode in $\F_q^n$. Hence, by \Cref{lem:gaussian-estimates},
\[
  |\LL_q(n,d)|
  \ge
  \sum_{r=0}^{d}\bqbinom d{r}
  \ge
  \bqbinom d{\lfloor d/2\rfloor}
  \ge
  q^{\lfloor d^2/4\rfloor} .
\]
Since $d=n-o(n)$, we have $\frac{d}{n}\to1$, and hence  $d^2=(1+o(1))n^2$.
Therefore,
\[
  q^{(1/4+o(1))n^2}
  \le
  |\LL_q(n,d)|
  \le
  q^{(1/4+o(1))n^2},
\]
which proves the claim.
\end{proof}

Combining the above estimates gives the main theorem on the number of linear $d$-anticodes:

\begin{theorem}\label{thm:linear-main}
Fix a prime power $q$, let \(n\to\infty\), and let \(0\le d=d(n)\le n\) be an integer. 
\begin{enumerate}
  \item If $d$ is fixed, then there is an integer $\sigma_q(d)$ satisfying $d\le \sigma_q(d)\le
    \left\lfloor\frac{qd}{q-1}\right\rfloor$ such that
  \[
    |\LL_q(n,d)|
    =
    \Theta_{q,d}\!\left(n^{\sigma_q(d)}\right). 
  \]
  In particular, if $q>d+1$, then $\sigma_q(d)=d$ and $|\LL_q(n,d)|
    =
    \Theta_{q,d}\!\left(n^{d}\right)$.

  \item If $d\to\infty$ and $d=o(n)$, then
  \[
    |\LL_q(n,d)|
    =
    q^{\Theta_q\!\left(d^2+d\log_q(n/d)\right)}.
  \]

 \item If $d/n\to\alpha$, where $0<\alpha<1$, and $\gamma=\min\left\{1,\frac{q\alpha}{q-1}\right\}$ then 
  \[
    q^{(\alpha^2/4+o(1))n^2}
    \le
    |\LL_q(n,d)|
    \le
    q^{(\gamma^2/4+o(1))n^2}.
  \]

  \item If $n-d=o(n)$, then
  \[
    |\LL_q(n,d)|
    =
    q^{(1/4+o(1))n^2}.
  \]
\end{enumerate}
\end{theorem}

\begin{proof}
We again treat the different diameter regimes separately. Note that -- contrary to the general anticode case in \Cref{lem:nonlinear-bulk} -- we distinguish between bounded diameters and diameters tending to infinity.

\paragraph{1. Fixed diameter.}
Fix $d$ and let $\sigma_q(d):=\max\{s:a_q(s,d)>0\}$.
Then $a_q(s,d)$ and hence $\sigma_q(d)$ are independent of $n$. 

Since $\F_q^d$ is a full-support linear $d$-anticode, we have
$a_q(d,d)>0$, and hence $\sigma_q(d)\ge d$. On the other hand, \Cref{lem:full-support-upper} implies that
$
  \sigma_q(d)\le \left\lfloor\frac{qd}{q-1}\right\rfloor.
$

If $q>d+1$, then the upper bound equals $d$, so $\sigma_q(d)=d$.

Moreover, by \Cref{lem:binomial-estimates}, we have
\[
  \binom ns=\frac{n^s}{s!}+O_s(n^{s-1}),
\]
which implies, together with \Cref{thm:linear-exact},
\begin{align*}
  |\LL_q(n,d)|
  &=\sum_{s=0}^{\sigma_q(d)}\binom ns a_q(s,d)=\sum_{s=0}^{\sigma_q(d)} \left(\frac{n^s}{s!}+O_s(n^{s-1})\right) a_q(s,d) \\
  &=\left(\frac{n^{\sigma_q(d)}a_q({\sigma_q(d)},d)}{{\sigma_q(d)}!}+O_{q,d}(n^{{\sigma_q(d)}-1}) \right)  
+\sum_{s=0}^{\sigma_q(d)-1} \left(\underbrace{\frac{n^s}{s!}}_{\leq n^{{\sigma_q(d)}-1}}+O_s(n^{s-1})\right) \underbrace{a_q(s,d)} _{\in O(1)} \\
  &=\underbrace{\frac{a_q(\sigma_q(d),d)}{\sigma_q(d)!}}_{\in \mathbb R_{> 0}}
    n^{\sigma_q(d)}+O_{q,d}(n^{\sigma_q(d)-1}) 
    = \Theta_{q,d}(n^{\sigma_q(d)})
\end{align*}

Since the constants can be chosen uniformly over any finite set of diameters, the same estimate holds for $d=d(n)\le D$, with implicit constants depending only on $q$ and $D$.

\paragraph{2. Sublinear diameter.} 
When $d$ grows with $n$, both the choice of support and the number of subspaces on that support contribute to the exponent.
Assume $q$ is fixed, $d\to\infty$, and $d=o(n)$. 
By \Cref{lem:coordinate-lower} and
\Cref{lem:gaussian-estimates},

\[
  |\LL_q(n,d)|
  \ge \Omega_q\!\left(\binom nd q^{d^2/4}\right).
\]
Since
\[
  \log_q\binom nd=(1+o(1))d\log_q\frac nd,
\]
we get
\[
  |\LL_q(n,d)|
  \ge q^{\Omega_q(d^2+d\log_q(n/d))}.
\]
By the upper bound in \Cref{thm:linear-exact} and \Cref{lem:gaussian-estimates},
\[
  |\LL_q(n,d)|
  \le
  \sum_{s=0}^{\lfloor qd/(q-1)\rfloor}
  \binom ns\sum_{r=0}^{s}\bqbinom sr \le
  O_q(d)\binom n{\lfloor qd/(q-1)\rfloor}
  q^{O_q(d^2+\log d)}.
\]
Also,
\[
  \log_q\binom n{\lfloor qd/(q-1)\rfloor}
  =O_q\!\left(d\log_q\frac nd\right).
\]
Therefore,
\[
  |\LL_q(n,d)|
  \le q^{O_q(d^2+d\log_q(n/d))}.
\]

\paragraph{3. Linear diameter.}

Fix $q$ and suppose $d/n\to\alpha$, where $0<\alpha<1$. 
Since $d^2=(\alpha^2+o(1))n^2$ and $\binom nd=2^{(h(\alpha)+o(1))n}$ contributes only $q^{o(n^2)}$, by \Cref{lem:coordinate-lower,lem:gaussian-estimates},
\[
  |\LL_q(n,d)|
  \ge \Omega_q\!\left(\binom nd q^{d^2/4}\right)
  =q^{(\alpha^2/4+o(1))n^2}.
\]
For the upper bound, put $\gamma=\min\left\{1,\frac{q\alpha}{q-1}\right\}$.
Using the support bound from \Cref{thm:linear-exact}, $\sum_{s\le n}\binom ns\le2^n$, and \Cref{lem:gaussian-estimates}, we get
\[
  |\LL_q(n,d)|
  \le 2^n q^{\gamma^2n^2/4+o(n^2)}
  =q^{(\gamma^2/4+o(1))n^2}.
\]

\paragraph{4. Near-full diameter $n-d=o(n)$.}

This follows from \Cref{prop:linear-near-full}.
\end{proof}

\begin{remark}
\Cref{thm:linear-exact} tells us that a better estimate of $a_q(s,d)$ will improve the bounds in the sublinear and linear-diameter regimes. 
However, for fixed $d$, the largest support size for which $a_q(s,d)\ne0$ determines the polynomial degree exactly.   
\end{remark}

\begin{remark}
The support bound in \Cref{lem:full-support-upper} is asymptotically tight as $d\to\infty$. To see this, consider the following construction.
Let $\Sim_q(k)\le \F_q^{\ell_k}$ be the $q$-ary simplex code, where $\ell_k=(q^k-1)/(q-1)$.  
It has full support and every non-zero codeword has weight $w_k=q^{k-1}$, since a non-zero linear functional vanishes on $(q^{k-1}-1)/(q-1)$ of the projective points in $\operatorname{PG}(k-1,q)$.  
Hence $\diam(\Sim_q(k))=w_k$.

Choose $k=k(d)$ with $w_k\to\infty$ and $w_k=o(d)$, put $m_k=\lfloor d/w_k\rfloor$ and $r_k=d-m_kw_k$, and define $C=\Sim_q(k)^{\oplus m_k}\oplus\F_q^{r_k}$ on disjoint coordinate blocks.  
Then $C$ has full support, diameter $d$, and support size $m_k\ell_k+r_k=(q/(q-1)+o(1))d$, because $\ell_k/w_k=q/(q-1)+o(1)$ and $r_k< w_k=o(d)$.
See \cite[Theorem~2.7.5]{HuffmanPless2003}.
\end{remark}

\begin{remark}
A direct lower bound follows by fixing a set $S\subseteq[n]$ of size $d$ and counting all subspaces supported inside $S$:
  $|\LL_q(n,d)|   \ge   \sum_{r=0}^{d}\bqbinom dr  \ge  q^{\lfloor d^2/4\rfloor}$.
We cannot simply multiply this count by $\binom nd$.
Indeed, a code with support size $t\le d$ is contained in $\binom{n-t}{d-t}$ coordinate spaces of dimension $d$.
Thus different choices of $S$ can produce the same code.
The full-support restriction in \Cref{lem:coordinate-lower} avoids this overcounting.

This construction parallels the general setting, where counting all subsets of a fixed $d$-anticode, such as a Hamming ball of radius $\lfloor d/2\rfloor$, gives a lower bound.
In both settings, the diameter condition is inherited by the objects being counted.
For general anticodes, counting subsets of one maximum anticode gives the correct leading asymptotic term in the regimes covered by \Cref{thm:nonlinear-main}.
The corresponding conclusion does not hold in general for linear anticodes: counting subspaces of one maximum linear anticode need not give the correct leading exponent.
The following binary construction demonstrates this by strictly improving the exponent $\alpha^2/4$ for every fixed $0<\alpha<1$.
\end{remark}

\subsection*{An improvement for $q=2$ and linear diameter}

For $d=\lfloor\alpha n\rfloor$, we use a random generator matrix on $s\sim d/\theta$ coordinates, where $\max\{\alpha,1/2\}<\theta<1$.
For every admissible choice of $\theta$, this construction gives a lower bound on the number of linear $d$-anticodes, with an exponent depending on $\theta$. 
We then optimize this exponent over $\theta$, obtaining a strict improvement over the exponent $\alpha^2/4$ in \Cref{thm:linear-main} for every $0<\alpha<1$.

For $1/2\le\theta\le1$, define
\[
  g(\theta)
  :=
  \begin{cases}
    h(\theta)(1-h(\theta)),
      & \text{if } h(\theta)\ge \frac12,\\
    \frac14,
      & \text{if } h(\theta)\le \frac12.
  \end{cases}
\]
For $0<\alpha<1$, define
\[
  \phi(\alpha)
  :=
  \max_{\max\{\alpha,1/2\}\le\theta\le1}
  \left(\frac{\alpha}{\theta}\right)^2g(\theta).
\]

\begin{proposition}\label{prop:binary-random-linear-lower}
Fix $0<\alpha<1$, and let $d=\lfloor\alpha n\rfloor$. Then
\[
  |\LL_2(n,d)|
  \ge
  2^{(\phi(\alpha)+o(1))n^2}.
\]
\end{proposition}

\begin{proof}
Let $d=\lfloor\alpha n\rfloor$. 
Choose
\[
  \max\{\alpha,1/2\}<\theta<1,
  \qquad
  0<\lambda<\min\left\{\frac12,1-h(\theta)\right\}.
\]
Let
\[
  s:=\left\lfloor\frac d\theta\right\rfloor,
  \qquad
  k:=\lfloor\lambda s\rfloor.
\]
Then $s=(\alpha/\theta+o(1))n$ and $s\le n$ for all sufficiently large $n$. 

Fix $S\subseteq[n]$ with $|S|=s$, and let $G$ be a uniformly random binary $k\times s$ matrix. 
For every $0\ne u\in\F_2^k$, $\wt(uG)\sim\operatorname{Bin}\left(s,\frac12\right)$.
Since $(d+1)/s=\theta+o(1)$, \Cref{lem:binomial-tail} and a union
bound give
\[
\begin{aligned}
  \mathbb P\bigl(\wt(uG)>d\bigr)
  &\le 2^{-(1-h(\theta)+o(1))s},\\
\implies  \mathbb P\bigl(\exists\,u\ne0:\wt(uG)>d\bigr)
  &\le 2^{(\lambda-(1-h(\theta))+o(1))s}
   =o(1).
\end{aligned}
\]

If $\rank(G)<k$, then there is some
$0\ne u\in\F_2^k$ such that $uG=0$.
For a fixed $u\ne0$, the coordinates of $uG$ are independent and hence $\mathbb P(uG=0)=2^{-s}$.

Taking a union bound over the $2^k-1$ non-zero choices of $u$ gives
\[
  \mathbb P\bigl(\rank(G)<k\bigr)
  \le
  (2^k-1)2^{-s}
  \le
  2^{k-s}
  =
  o(1).
\]
Also, a fixed column of $G$ is zero with probability $2^{-k}$, so
\[
  \mathbb P\bigl(\text{$G$ has a zero column}\bigr)
  \le
  s2^{-k}
  =
  o(1).
\]
Thus, with probability $1-o(1)$, the row span of $G$ is a
$k$-dimensional full-support binary linear $d$-anticode on $S$.

Every $k$-dimensional code is generated by as many full-rank
$k\times s$ matrices as there are ordered bases of the code.
The number of such ordered bases is exactly
\[
  (2^k-1)(2^k-2)\cdots(2^k-2^{k-1})
  \le 2^{k^2}.
\]
Therefore the number of resulting codes is at least
\[
  \frac{(1-o(1))2^{ks}}{2^{k^2}}.
\]
Consequently for every fixed admissible pair $(\theta,\lambda)$ we have,
\[
 \log_2|\LL_2(n,d)|
  \ge ks-k^2-o(1)
  \ge
  \left(
    \left(\frac{\alpha}{\theta}\right)^2
    \lambda(1-\lambda)+o(1)
  \right)n^2.
\]
Therefore,
\[
\liminf_{n\to\infty}
\frac{\log_2|\LL_2(n,\lfloor\alpha n\rfloor)|}{n^2}
\ge
\left(\frac{\alpha}{\theta}\right)^2\lambda(1-\lambda).
\]
Taking the supremum over
$0<\lambda<\min\{1/2,1-h(\theta)\}$ gives
$(\alpha/\theta)^2g(\theta)$.
Taking the supremum over
$\max\{\alpha,1/2\}<\theta<1$ and using continuity gives
$\phi(\alpha)$.
Therefore,
\[
  |\LL_2(n,d)|
  \ge
  2^{\left((\alpha/\theta)^2g(\theta)+o(1)\right)n^2}
\]
for every $\max\{\alpha,1/2\}<\theta<1$.

For $\theta=1$, take $s=d$, $r=\lfloor s/2\rfloor$, and let
$N_s$ be the number of full-support $r$-dimensional subspaces of
$\F_2^s$. Then
\[
  \bqbinom sr-s\bqbinom{s-1}r
  \le N_s\le\bqbinom sr.
\]
By \Cref{lem:gaussian-estimates},
\[
  N_s=\Theta\!\left(2^{s^2/4}\right),
  \qquad
  |\LL_2(n,d)|
  \ge
  2^{(\alpha^2/4+o(1))n^2}.
\]

By continuity, the lower endpoint is obtained by taking $\theta$
towards $\max\{\alpha,1/2\}$, while the case $\theta=1$
was handled separately. Taking the supremum over all admissible
$\theta$ proves the result.

\end{proof}


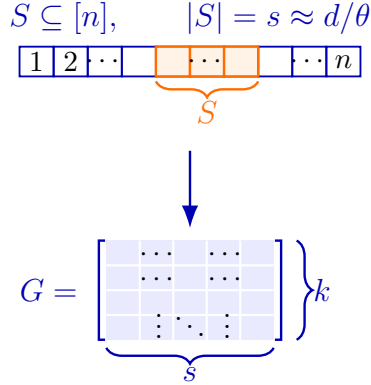
\begin{figure}[htbp]
\centering
\makebox[\linewidth][c]{%
\begin{tikzpicture}[
    scale=0.82,
    every node/.style={font=\normalsize},
    arrow/.style={
        -{Latex[length=3mm,width=2.2mm]},
        line width=0.9pt,
        blue!70!black
    }
]

\colorlet{mainblue}{blue!70!black}
\colorlet{softblue}{blue!18}
\colorlet{mainorange}{orange!85!red}

\def\drawbrackets#1#2#3#4{%
    \draw[mainblue,line width=0.9pt]
        ({#1+0.10},{#2+#4})
        -- (#1,{#2+#4})
        -- (#1,#2)
        -- ({#1+0.10},#2);
    \draw[mainblue,line width=0.9pt]
        ({#1+#3-0.10},{#2+#4})
        -- ({#1+#3},{#2+#4})
        -- ({#1+#3},#2)
        -- ({#1+#3-0.10},#2);
}


\node[text=mainblue,font=\large] at (0,5.9) {
    $\displaystyle
    S\subseteq[n],
    \qquad |S|=s\approx d/\theta
    $
};

\def\xstart{-2.75}
\def\ystart{4.95}
\def\w{0.55}
\def\h{0.48}

\foreach \i in {0,...,9} {
    \draw[mainblue,line width=0.7pt]
        ({\xstart+\i*\w},\ystart)
        rectangle ++(\w,\h);
}

\foreach \i in {4,5,6} {
    \draw[mainorange,fill=orange!10,line width=0.9pt]
        ({\xstart+\i*\w},\ystart)
        rectangle ++(\w,\h);
}

\foreach \i/\lab in {
    0/1,1/2,2/\cdots,5/\cdots,8/\cdots,9/n
} {
    \node at ({\xstart+(\i+0.5)*\w},{\ystart+0.24})
        {$\lab$};
}

\draw[
    decorate,
    decoration={brace,mirror,amplitude=6pt},
    mainorange,
    line width=0.9pt
]
    ({\xstart+4*\w},{\ystart-0.08})
    -- ({\xstart+7*\w},{\ystart-0.08});

\node[text=mainorange,font=\large]
    at ({\xstart+5.5*\w},{\ystart-0.58}) {$S$};

\draw[arrow] (0,3.75) -- (0,2.5);


\node[text=mainblue,font=\large,anchor=east]
    at (-1.65,1.50) {$G=$};

\def\mx{-1.35}
\def\my{0.70}
\def\mw{2.70}
\def\mh{1.60}

\fill[softblue!45]
    (\mx,\my) rectangle ++(\mw,\mh);

\foreach \i in {1,...,4} {
    \draw[white,line width=0.8pt]
        ({\mx+\i*\mw/5},\my)
        -- ({\mx+\i*\mw/5},{\my+\mh});
}

\foreach \i in {1,2,3} {
    \draw[white,line width=0.8pt]
        (\mx,{\my+\i*\mh/4})
        -- ({\mx+\mw},{\my+\i*\mh/4});
}

\drawbrackets{\mx-0.12}{\my}{\mw+0.24}{\mh}

\foreach \y in {2.07,1.67} {
    \node at (-0.50,\y) {$\cdots$};
    \node at ( 0.60,\y) {$\cdots$};
}

\node at (-0.50,1.05) {$\vdots$};
\node at ( 0.00,1.05) {$\ddots$};
\node at ( 0.60,1.05) {$\vdots$};

\draw[
    decorate,
    decoration={brace,mirror,amplitude=6pt},
    mainblue,
    line width=0.9pt
]
    (\mx,0.57) -- ({\mx+\mw},0.57);

\node[text=mainblue,font=\large] at (0,0.15) {$s$};

\draw[
    decorate,
    decoration={brace,mirror,amplitude=6pt},
    mainblue,
    line width=0.9pt
]
    (1.72,\my) -- (1.72,{\my+\mh});

\node[text=mainblue,font=\large] at (2.13,1.50) {$k$};

\end{tikzpicture}%
}

\caption{$C=\operatorname{rowspan}(G)$ is a linear $d$-anticode of dimension-$k$ w.h.p.}
\label{fig:probabilistic-construction}
\end{figure}

\begin{remark}\label{rem:binary-random-improvement}
The bound in \Cref{prop:binary-random-linear-lower} strictly improves
the exponent $\alpha^2/4$ in \Cref{thm:linear-main}, for every
$0<\alpha<1$. Indeed, choose $\theta<1$ sufficiently close to $1$
so that $\theta>\alpha$ and $h(\theta)<1/2$. Then
\[
  \left(\frac{\alpha}{\theta}\right)^2g(\theta)
  =
  \frac{\alpha^2}{4\theta^2}
  >
  \frac{\alpha^2}{4}.
\]
\end{remark}
\begin{corollary}\label{cor:binary-exact-exponent}
Let $\alpha_0\in(1/2,1)$ be the unique solution of $h(\alpha_0)=\frac12$.
For every fixed $\alpha_0\le\alpha<1$,
\[
  |\LL_2(n,\lfloor\alpha n\rfloor)|
  =
  2^{(1/4+o(1))n^2}.
\]
\end{corollary}

\begin{proof}
If $\alpha\ge\alpha_0$, then every admissible $\theta\ge\alpha$
satisfies $h(\theta)\le1/2$, and hence $g(\theta)=1/4$. Therefore
\[
  \phi(\alpha)
  =
  \max_{\alpha\le\theta\le1}
  \frac{\alpha^2}{4\theta^2}
  =
  \frac14.
\]
Thus \Cref{prop:binary-random-linear-lower} gives
\[
  |\LL_2(n,\lfloor\alpha n\rfloor)|
  \ge
  2^{(1/4+o(1))n^2}.
\]
On the other hand, for $q=2$ and $\alpha\ge1/2$, the parameter in
\Cref{thm:linear-main} is $\gamma=\min\{1,2\alpha\}=1$, so
\[
  |\LL_2(n,\lfloor\alpha n\rfloor)|
  \le
  2^{(1/4+o(1))n^2}.
\]
\end{proof}

\begin{remark}
Every subspace of a code produced by the above construction is again a linear $d$-anticode.
This introduces an overcounting issue: the same subspace may be contained in several distinct constructed codes.
For the parameters used here, counting these additional subspaces cannot improve the leading exponent.
Indeed, $k\le s/2$, and \Cref{lem:gaussian-uniform} gives the following.
\[
  \sum_{r=0}^{k}\genfrac{[}{]}{0pt}{}{s}{r}_2
  \le
  C_2(k+1)2^{k(s-k)}.
\]
Thus even all subspaces of dimension at most $k$ on the
fixed set of $s$ coordinates have the same leading
exponent as the family already counted.
\end{remark}

\section{Conclusion, further discussion, and open problems }\label{sec:further-directions}

We conclude this paper by explaining the broader implications of our counts, stating some related results on the number of anticodes, and some ideas for further research.

\subsection{Broader implications}

Our results reveal an interesting difference between the asymptotic behavior of general and linear anticodes. 
In the general setting, for $4\le d\le n-1$, excluding the binary endpoint $q=2$ and $d=n-1$, the asymptotic number of
anticodes of a prescribed diameter is governed, on the logarithmic scale, by the number of subsets of one maximum anticode.
In other words, counting these subsets already gives the correct leading term in the logarithm of the total count. 
As a consequence, the number of inclusion-maximal anticodes is asymptotically negligible compared with the total number of anticodes.

The situation is different in the linear setting. 
Counting subspaces of one maximum linear anticode need not give the correct leading term on the logarithmic scale.
We organize all linear anticodes according to their support, using the one-to-one correspondence between codes with a prescribed support and full-support codes on those coordinates.
This decomposition gives our asymptotic estimates and bounds.
These contrasting behaviors demonstrate that imposing linearity fundamentally changes the asymptotic growth of the number of anticodes.

\subsection{Asymptotics for growing $q$}

The asymptotic results above keep $q$ fixed while $n\to\infty$.
It is natural to ask what happens when we fix $n$ and let the field size grow.

\paragraph*{General anticodes}
For general anticodes, the Helly bound in \Cref{thm:AJS} does not depend on $q$.  
As a first simple case,
fix $n$ and $1\le d<n$, and let $q\to\infty$.  
By
\Cref{lem:maximal-count} and $|B(x,d)|\le q^n$, $$\log_2|\mM_q(n,d)|=  O_{n,d}(\log q).$$
On the other hand, \Cref{lem:coordinate-subcube-lower} gives $m_q(n,d)\ge q^d$. 
Therefore, we get the same result as in Section $\ref{sec:nonlinear}$, namely
$$|\mA_q(n,d)|=2^{(1+o_q(1))m_q(n,d)},$$
where the subscript in $o_q$ indicates that the asymptotics are for growing $q$. 

\paragraph*{Linear anticodes}
For the linear anticode count, as $q\to\infty$, the possible support size approaches $d$.
For fixed $n$ and $d$, this gives
\[
  \log_q|\LL_q(n,d)|
  =
  \left\lfloor\frac{d^2}{4}\right\rfloor+o_q(1).
\]
If $q=q(n)\to\infty$ through prime powers and
$d=\lfloor\alpha n\rfloor$ with fixed $0<\alpha<1$, we get
\[
  |\LL_q(n,d)|
  =
  q^{(\alpha^2/4+o(1))n^2}.
\]
The support bound gives $|\supp(C)|  \le   \left(\alpha+o(1)\right)n$ and the lower bound is again obtained from subspaces supported on $d$ coordinates.
Thus, the gap in the exponents of the lower and the upper bounds in linear regime, disappears when
$q\to\infty$.

\paragraph*{Open problem}
The next interesting task would be to determine the general count when $q=q(n)$ and $d=d(n)$ vary simultaneously. We leave this for future research:
\begin{problem}
Determine the asymptotic behavior for $|\mA_{q(n)}(n,d(n))|$ and $|\LL_{q(n)}(n,d(n))|$ when the field size $q$ and the diameter $d$ grow simultaneously.
\end{problem}

\subsection{Restricted distances}
An extension of the problems considered in this paper is to impose simultaneously a lower and an upper bound on all pairwise distances.
For integers $0\le d_1\le d_2\le n$, define
\[
    \mathcal C_q(n;d_1,d_2)
    :=
    \left\{
        C\subseteq \F_q^n : |C|\ge2, \;
        d_1\le \dist(x,y)\le d_2
        \text{ for all distinct }x,y\in C
    \right\}.
\]
Equivalently, every $C\in\mathcal C_q(n;d_1,d_2)$ satisfies
\[
    d_{\min}(C)\ge d_1
    \qquad\text{and}\qquad
    \diam(C)\le d_2.
\]
Thus, this class sits between the usual minimum-distance condition for error-correcting codes and the maximum-distance condition that defines anticodes.
Codes satisfying such a condition may be viewed as codes with \emph{restricted distances} or \emph{interval-distance codes}. 

The  analogous object to study for linear codes is
\[
    \mathcal L_q(n;d_1,d_2)
    :=
    \left\{
        0 \ne C\le\F_q^n :
        d_1\le \wt(c)\le d_2
        \text{ for every }0\ne c\in C
    \right\}.
\]
Indeed, for a linear code $C$, $d_{\min}(C) = \min_{0\ne c\in C}\wt(c),$ and $\diam(C)=\max_{0\ne c\in C}\wt(c)$.

We pose the following enumeration problems for future research: 

\begin{problem}\;
\begin{enumerate}
    \item 
Determine the asymptotic behavior of $|\mathcal C_q(n;d_1,d_2)|$.
    \item
Determine the asymptotic behavior of $|\mathcal L_q(n;d_1,d_2)|$.
\end{enumerate}
\end{problem}
\medskip

\textbf{Acknowledgments.}
The second author thanks Nikolaas D.~Verhulst for helpful discussions during the early stages of this project.
The random-generator construction used for binary linear anticodes originated from an idea suggested by Noga Alon to Anurag Bishnoi in an email correspondence approximately two years earlier, in connection with a different problem concerning vector-space Ramsey numbers.

\textbf{AI declaration.} 
The second author used AI (various models) for literature searches.
It brought the work of Alon, Jin, and Sudakov~\cite{AlonJinSudakov2024} on the Helly number of Hamming balls to the author's attention.
One of its suggestions was to use the support of a linear anticode as the organizing parameter for its enumeration. 
Furthermore, AI models were used for language correction and stylistic improvements.


\medskip

\begin{thebibliography}{99}

\bibitem{AhlswedeKatona1977}
R. Ahlswede and G.~O.~H. Katona,
\newblock Contributions to the geometry of Hamming spaces,
\newblock \emph{Discrete Math.} \textbf{17} (1977), no.~1, 1--22.

\bibitem{AhlswedeKhachatrian1998}
R.~Ahlswede and L.~H. Khachatrian,
\newblock The diametric theorem in Hamming spaces---optimal anticodes,
\newblock \emph{Advances in Applied Mathematics} \textbf{20} (1998), no.~4, 429--449.

\bibitem{AhlswedeAydinianKhachatrian2001}
R.~Ahlswede, H.~K. Aydinian, and L.~H. Khachatrian,
\newblock On perfect codes and related concepts,
\newblock \emph{Designs, Codes and Cryptography}
\textbf{22} (2001), no.~3, 221--237.

\bibitem{AlonJinSudakov2024}
N.~Alon, Z.~Jin, and B.~Sudakov,
\newblock The Helly number of Hamming balls and related problems,
\newblock arXiv:2405.10275, 2024.

\bibitem{moduleDensity}{
E. Byrne, A.-L. Horlemann, K. Khathuria, and V. Weger,
\newblock Density of Free Modules over Finite Chain Rings,
\newblock \emph{Linear Algebra and its Applications}, \textbf{651} (2022), 1--25.
}

\bibitem{Das2016Binomial}
S.~Das,
\newblock A brief note on estimates of binomial coefficients,
\newblock
\url{https://discretemath.imp.fu-berlin.de/DMI-2016/notes/binomials.pdf}.

\bibitem{DongManiZhao2023}
D.~Dong, N.~Mani, and Y.~Zhao,
\newblock On the number of error correcting codes,
\newblock \emph{Combinatorics, Probability and Computing}
\textbf{32} (2023), no.~5, 819--832.

\bibitem{Frankl2017}
P. Frankl,
\newblock A stability result for families with fixed diameter,
\newblock \emph{Combin. Probab. Comput.} \textbf{26} (2017),
no.~4, 506--516.

\bibitem{DiGiustoGorlaRavagnani2025}
A.~Di Giusto, E.~Gorla, and A.~Ravagnani,
\newblock A distance-free approach to generalized weights,
\newblock arXiv:2512.17542, 2025.

\bibitem{densities}{
A. Gruica, A.-L. Horlemann, A. Ravagnani, and N. Willenborg,
\newblock Densities of Codes of Various Linearity Degrees in Translation-Invariant Metric Spaces,
\newblock \emph{Designs, Codes and Cryptography}, \textbf{92} (2024), 609--637.
}

\bibitem{HallNotes}
J.~I. Hall,
\newblock \emph{Notes on Coding Theory}.
\newblock 2003.

\bibitem{HuffmanPless2003}
W.~C. Huffman and V.~Pless,
\newblock \emph{Fundamentals of Error-Correcting Codes},
\newblock Cambridge University Press, 2003.


\bibitem{Ihringer2015}
F.~Ihringer,
\newblock \emph{Finite Geometry Intersecting Algebraic Combinatorics:
An Investigation of Intersection Problems Related to
Erd\H{o}s--Ko--Rado Theorems on Galois Geometries with Help from
Algebraic Combinatorics},
\newblock Ph.D. thesis, Justus-Liebig-Universit\"at Gie{\ss}en, 2015.


\bibitem{Kleitman1966}
D.~J. Kleitman,
\newblock On a combinatorial conjecture of Erd\H{o}s,
\newblock \emph{Journal of Combinatorial Theory} \textbf{1} (1966), no.~2, 209--214.

\bibitem{Ravagnani2016}
A.~Ravagnani,
\newblock Generalized weights: an anticode approach,
\newblock \emph{Journal of Pure and Applied Algebra} \textbf{220} (2016), no.~5, 1946--1962.
\newblock arXiv:1410.7207.

\bibitem{tLeeCodes}{
N. Willenborg, A.-L. Horlemann, and V. Weger,
\newblock On the Number of t-{L}ee-error-correcting Codes,
\newblock \emph{Advances in Mathematics of Communications}, vol. \textbf{18} (2024), no.~2, 567--594.
} 

\bibitem{Zhao2010}
Y.~Zhao,
\newblock The number of independent sets in a regular graph,
\newblock \emph{Combin. Probab. Comput.} \textbf{19} (2010), no.~2, 315--320.

\end{thebibliography}
\end{document}